\documentclass{article}
\usepackage[margin=1in]{geometry}
\usepackage{graphicx}
\usepackage[T1]{fontenc}

\usepackage{amsmath}
\usepackage{amsfonts}
\usepackage{amssymb}
\usepackage{amsbsy}
\usepackage{amsthm}
\usepackage[algo2e,ruled]{algorithm2e}
\usepackage{bbm}
\usepackage{xcolor}

\def\eps{\varepsilon}

\def\SC{\mathsf{TOC}}
\def\E{\mathbb{E}}

\usepackage[pdftex,colorlinks=true,urlcolor=blue,citecolor=black,anchorcolor=black,linkcolor=black]{hyperref}
\usepackage{cleveref}

\def\sF{\mathcal{F}}
\newcommand{\beqn}[1]{\begin{equation}\label{#1}}
	\newcommand{\eeqn}{\end{equation}}
\def\EE{\mathbb{E}}

\def\R{\mathbb{R}}
\def\P{\mathbb{P}}

\newcommand{\one}{{\mathbbm{1}}}

\newcommand{\kval}{{C\max\{\frac{1}{\lambda},\frac{\ln(2)}{a}\}, \text{ $C$ is a universal constant}}}

\newcommand\norm[1]{\left\lVert{#1}\right\rVert}

\newtheorem{theorem}{Theorem}

\newtheorem{corollary}[theorem]{Corollary}

\newtheorem{definition}{Definition}

\newtheorem{remark}{Remark}

\newtheorem{lemma}{Lemma}

\newtheorem{assumption}{Assumption}

\newtheorem{proposition}{Proposition}

\newcommand{\blue}[1]{{\color{blue}{#1}}}

\begin{document}

%
%

\title{First- and Second-Order Stochastic Adaptive Regularization with Cubics: High Probability Iteration and Sample Complexity\footnotemark[1]}

\author{Katya Scheinberg\footnotemark[2]
   \and Miaolan~Xie\footnotemark[3]}

\footnotetext[1]{A preliminary version of this work appeared in the Proceedings of the 2023 Winter Simulation Conference under the title ``Stochastic Adaptive Regularization Method with Cubics: A High-Probability Complexity Bound" \cite{scheinberg2023stochastic}. That version contains the first-order iteration complexity result. 
The current paper presents an extended and unified treatment, including the second-order iteration complexity and both the first- and second-order sample complexity results.}
\footnotetext[2]{School of Industrial and Systems Engineering, Georgia Tech, Atlanta, GA, USA; E-mail: katyascheinberg@gmail.com}
\footnotetext[3]{Edwardson School of Industrial Engineering, Purdue University, 
West Lafayette, IN, USA; E-mail: xie537@purdue.edu}

\maketitle

\begin{abstract}
 We present high-probability (and expectation) complexity bounds for two versions of stochastic adaptive regularization methods with cubics (SARC), also known as regularized Newton methods. The first algorithm aims to find first-order stationary points, while the second targets second-order optimality conditions. Both methods employ stochastic zeroth-, first-, and second-order oracles with specific accuracy and reliability requirements. These oracles, which have been previously used with other stochastic adaptive methods like trust-region and line-search algorithms, are applicable to various optimization settings including expected risk minimization and simulation optimization. In this paper, we establish the first high-probability iteration and sample complexity bounds for both first- and second-order SARC algorithms. Our analysis demonstrates that as in the deterministic case, they outperform other stochastic adaptive methods.
\end{abstract}

\section{Introduction}

We are interested in unconstrained optimization problems of the form
	\[
	\min_{x \in \mathbb{R}^m} \phi(x),
	\]
	where $\phi$ is possibly nonconvex and satisfies the following conditions:  	



	\begin{assumption}\label{ass:Lip_nablaf}\quad$\phi$ is  bounded from below by a constant $\phi^*$, is twice continuously differentiable, and has globally $L$-Lipschitz continuous gradient and  $L_H$-Lipschitz continuous  Hessian.
	\end{assumption}
	We present and analyze two versions of stochastic adaptive cubic regularization algorithms. The first is designed to compute a point $x$ such that $\norm{\nabla \phi(x)}\leq \epsilon$, for some $\epsilon>0$. The second is designed to compute a point $x$ satisfying both $\left\|\nabla \phi\left(x\right)\right\| \leq \epsilon$ and $\lambda_{\text{min}}(\nabla^2\phi(x)) \geq -\sqrt{\epsilon}$, where $\lambda_{\min}(\cdot)$ is the minimum eigenvalue of a matrix. Both algorithms are under the setting that $\phi(x)$ or its derivatives are not computable exactly. Specifically, we assume  access to stochastic zeroth-, first- and second-order oracles, which are defined as follows. 	
   
	{\bf Stochastic zeroth-order oracle} ($\mathsf{SZO}(\epsilon_f, \lambda,a)$). 
	Given  a point $x$,  the oracle computes  $f(x,\Xi(x))$, 
	where $\Xi(x)$ is a random variable, whose distribution may depend on $x$, $\epsilon_f, \lambda$ and $a$, that satisfies 
	\begin{equation}\label{eq:zero_order}
	{\mathbb E_{\Xi(x)}}\left [ |{\phi(x) - f(x, \Xi(x))}|\right ]\leq \epsilon_f \ \text{~and~ }
	{\mathbb P_{\Xi(x)}}\left ( |{\phi(x) - f(x, \Xi(x))}|< t\right ) \geq 1-e^{\lambda (a-t)}, 
	\end{equation}
	 for any $t>0$.

We view  $x$ as the input to the oracle,  $f(x,\Xi(x))$ as the output and the values $(\epsilon_f, \lambda, a)$ as values intrinsic to the oracle. 
Thus $|f(x,\Xi(x))-\phi(x)|$ is a sub-exponential random variable with parameters 
	$(\lambda,a)$, whose mean is bounded by some constant $\epsilon_f>0$. 
Note that the definition of $\mathsf{SZO}$ implies that 
\begin{equation}\label{eq:zero_order2}
	\ {\mathbb E_\Xi}\left [\exp\left\{\tau (E(x)-\EE[E(x)])\right\}\right ]\leq \exp\left(\frac{\tau^2K^2}{2}\right), \quad \forall \tau\in \left[0,\frac{1}{K}\right],
	\end{equation}	
	where $K=\kval$. This follows from the sub-exponential property \eqref{eq:zero_order} by applying Proposition 2.7.1 of \cite{vershynin2018high}.

	
	{\bf Stochastic first-order oracle} ($\mathsf{SFO}(\kappa_g)$.
	Given a point $x$ and constants $\mu_1> 0$, $\delta_1\in [0,\frac 1 2)$,
	the oracle computes $g(x,\Xi^1(x))$, such that
	\begin{equation}\label{sfo}
	\P_{\Xi^1(x)}(\|\nabla \phi(x)-g(x,\Xi^1(x))\|\leq  \kappa_g\mu_1)\geq 1-\delta_1,
	\end{equation}
	where $\Xi^1(x)$ is a random variable whose distribution may depend on $x$, $\mu_1$, $\delta_1$ and $\kappa_g$.
	 We view $x$, $\mu_1$ and $\delta_1$ as inputs to the oracle, while $\kappa_g$ is intrinsic to the oracle. 
	
	{\bf Stochastic second-order oracle} ($\mathsf{SSO}(\kappa_H)$). Given a point $x$ and constants $\mu_2> 0$, $\delta_2\in [0,\frac 1 2)$,
	the oracle computes $H(x,\Xi^2(x))$, such that
\begin{equation}\label{sso}
\P_{\Xi^2(x)}(\|\nabla^2 \phi(x)-{H}(x,\Xi^2(x))\|\leq  \kappa_H\mu_2)\geq 1-\delta_2,
\end{equation}
where $\Xi^2(x)$ is a random variable whose distribution may depend on $x$, $\mu_2$, $\delta_2$ and $\kappa_H$. The norm on the matrix is the operator norm.
$x$, $\mu_2$ and $\delta_2$ are inputs to the oracle, while $\kappa_H$ is intrinsic to the oracle. 	

Wherever possible, we will omit the dependence on $x$ and write $\Xi, \Xi^1$, and $\Xi^2$ instead of $\Xi(x), \Xi^1(x)$, and $\Xi^2(x)$, and we will use $f(x), g(x)$ and $H(x)$ to denote the outputs of the stochastic oracles  for brevity.

\medskip\noindent\textbf{Related work.} 

\blue{
}

Several stochastic adaptive optimization methods have been studied in recent years under various stochastic oracle assumptions, similar to the ones 
we present above. Specifically, \cite{cartis2018global} bounds the expected iteration complexity for an adaptive step search (line search) method and an adaptive cubic regularization method, with a similar first-order oracle, but with an exact zeroth-order oracle. In \cite{paquette2020stochastic},  an expected iteration complexity result is derived for a variant of a step search method under a stochastic zeroth-order oracle. In \cite{jinhigh}, the results of  \cite{paquette2020stochastic} are strengthened under a somewhat more restrictive zeroth-order oracle, which is similar to the $\mathsf{SZO}$ in this paper, and a high probability complexity bound is derived. 

Similarly, in \cite{bandeira2014convergence},  \cite{gratton2018complexity}, a trust region method is analyzed under essentially $\mathsf{SFO}$ and $\mathsf{SSO}$, but with an exact zeroth-order oracle. Later in \cite{STORM,blanchet2019convergence} an expected complexity bound is derived for a trust region method with a stochastic zeroth-order oracle. 
In \cite{cao2022first} a high probability iteration complexity bound for first- and second-order stochastic trust region methods is derived under  the same oracles we use. 
Recently, the same oracles were used within a stochastic quasi-Newton method in \cite{qNewton}, and a stochastic SQP-based method for nonlinear equality constrained problems in \cite{sqp}.

Adaptive regularization with cubics (ARC) methods are theoretically superior alternatives to line search and trust region methods, when applied to 
deterministic smooth functions, 
 because of their optimal complexity of $O(\epsilon^{-3/2})$ vs $O(\epsilon^{-2})$ for finding $\epsilon$ stationary points \cite{cagotoPII,cgt38}. 
 There are many  variants of adaptive cubic regularization methods under various  assumptions and requirements on the function value, gradient, and Hessian estimates. 
Specifically, in \cite{cagotoPI,liu2018stochastic,bellavia2019adaptive,wang2019note,kohler2017sub,park2020combining}, the oracles are assumed to be deterministically bounded, with adaptive  magnitude or errors. 
In \cite{cartis2018global,bellavia2022stochastic,bellavia2022adaptive,bellavia2020stochastic}, bounds on expected complexity are provided under 
exact or deterministically  bounded zeroth-order oracle and the  gradient and Hessian oracles similar to $\mathsf{SFO}$ and $\mathsf{SSO}$. 
A cubicly regularized method in a fully stochastic setting is  analyzed  in \cite{tripuraneni2018stochastic}. This non-adaptive method relies on known Lipschitz constants of \(\nabla \phi(x)\) and therefore does not require a zeroth-order oracle.
It uses stochastic gradient and Hessian-vector products to find $\epsilon$-approximate first-order stationary points and $(\epsilon, \sqrt{\epsilon})$-approximate second-order stationary points for smooth, nonconvex functions in $\tilde{O}(\epsilon^{-3.5})$ evaluations.
However, the results assume sufficiently accurate stochastic gradient and Hessian estimates at each iteration. The final complexity bound only applies with probability that this holds true. No expected complexity bound can thus be derived. 

In addition to the methods discussed earlier, several works have explored cubic regularization methods that operate in random subspaces. These methods, such as those proposed by \cite{hanzely2020stochastic, zhaocubic, cartis2025random}, typically require exact evaluations of function values, gradients, or Hessians. Additionally, there are stochastic variance reduction-based cubic regularization methods \cite{pmlr-v80-zhou18d, wang2019stochastic, zhou2020stochastic, zhang2022adaptive} that are applicable to finite-sum functions. The final complexity bounds of these methods depend on the total number of terms in the finite sum.


\medskip\noindent\textbf{Our contributions.}
In this work we provide the first comprehensive analysis of a stochastic Adaptive regularization with cubics method (SARC) that establishes both first- and second-order iteration complexity bounds and sample complexity bounds with high probability and in expectation. Our method allows for: 
\begin{enumerate}
\item Stochastic function value estimates that can have arbitrarily large errors, and 
\item Stochastic gradient and Hessian approximations whose error is bounded by an adaptive quantity with sufficiently high probability,  but otherwise can be arbitrarily large. 
\end{enumerate}
Our work is the first to establish such first- and second-order iteration complexity bounds that match the deterministic complexity of \(O(\epsilon^{-3/2})\) in this setting with an overwhelmingly high probability. We demonstrate that the stopping times where an approximate stationary point is found are sub-exponential random variables, scaling as \(O(\epsilon^{-3/2})\).
The high probability sample complexity bounds we establish are also novel. We show that our variants of stochastic ARC, while more general than those in prior literature, still maintains its optimal iteration complexity.

The analysis presented here extends the stochastic settings and high probability results in \cite{jinhigh} and \cite{cao2022first} to the framework in \cite{cartis2018global}. However, it requires careful modification of most of the elements of the existing analysis. We point out these modifications in the appropriate places in the paper. 

The oracles used in this paper are essentially the same as in \cite{jinhigh} and \cite{cao2022first}. However, our assumption on the oracles is a bit stronger in this paper than in these two previous works.  In particular, we assume that  $\mathsf{SFO}$ and $\mathsf{SSO}$ are implementable for arbitrarily small values of $\mu_1$ and $\mu_2$, respectively. In contrast, the analysis in  \cite{jinhigh} and \cite{cao2022first} allows for the case when 
these oracles cannot be implemented for arbitrarily small error bounds. We will discuss further in the paper, that even though SARC may impose small values of $\mu_1$ and $\mu_2$, this happens only with small probability. 

We do not discuss the numerical performance of our method in this paper. Although deterministic implementations of ARC can be competitive with trust-region and line search methods when implemented with care, their efficiency in practice is highly dependent on the subproblem solver used. We expect similar behavior in the stochastic case and leave this study as a subject of future research. 

\medskip\noindent\textbf{Paper structure.} 
The rest of the paper is organized as follows:
In Section \ref{sec:first_order}, we introduce the first-order SARC algorithm and analyze its high probability iteration complexity bound.
In Section \ref{sec:second_order}, we extend the analysis to second-order complexity bounds. We introduce the necessary modifications to the algorithm and provide the second-order iteration complexity.
In Section \ref{sec:sample_complexity}, we derive the high probability sample complexity bounds for both the first- and second-order SARC algorithms.

\section{First-Order SARC}\label{sec:first_order}
	The first-order Stochastic Adaptive Regularization with Cubics  (SARC) method is presented below as Algorithm \ref{alg:ARC_Random}.
	At each iteration $k$,  given gradient estimate $g_k$, Hessian estimate $H_k$, and a regularization parameter  $\sigma_k>0$,  the following model  is approximately minimized with respect to $s$ to obtain the trial step $s_k$:
	\begin{equation}\label{cubic}
	m_k(x_k+s)= \phi(x_k)+ s^Tg_k+\frac{1}{2}s^TH_ks+\frac{\sigma_k}{3}\|s\|^3.
	\end{equation}
	The constant term $\phi(x_k)$  is never computed and is used simply for presentation purposes. In the case of SARC, $g_k$ and $H_k$ are computed  using $\mathsf{SFO}$ and $\mathsf{SSO}$ so as to satisfy certain accuracy requirements with sufficiently high probability, which will be specified in \Cref{sec:det_props}. 
 We require the trial step $s_k$ to be an ``approximate minimizer" of $m_k(x_k + s)$ in the sense that it has to satisfy:
		\begin{equation}\label{s-calc}
		(s_{k})^Tg_{k}+(s_{k})^TH_{k}s_{k}+\sigma_k\|s_{k}\|^3=0\,\,{\rm
			and}\,\,(s_{k})^TH_{k}s_{k}+\sigma_k\|s_{k}\|^3\geq 0
	\end{equation}
		and
	\begin{equation}\label{TCs}
		\|\nabla m_{k}(x_k+s_{k})\|\leq \eta \min\left\{1,\|s_{k}\|\right\} \|g_{k}\|,
	\end{equation}	
	where $\eta\in (0,1)$ is a user-chosen constant. The conditions are typical in the literature, e.g., in \cite{cagotoPII}
	 and can be satisfied, for example, using algorithms in \cite{cagotoPI,carmon2019gradient} to approximately minimize the model \eqref{cubic}, as well as by any global minimizer of $m_k(x_k+s)$.

As in any variant of the ARC method, once $s_k$ is computed,  the trial step  is accepted (and $\sigma_k$ is decreased) if the estimated function value of $x_k^+=x_k+s_k$  is sufficiently smaller than that of $x_k$, 
when compared to the model value decrease. We call these iterations \emph{successful}. Otherwise, the trial step is rejected (and $\sigma_k$ is increased). We call these iterations \emph{unsuccessful}.
 In the case of SARC, however, function value estimates are obtained via  $\mathsf{SZO}$ and the step acceptance criterion is modified by adding an ``error correction" term of $2\epsilon_f'$. This is because function value estimates have an irreducible error, so without this correction term, the algorithm may always reject improvement steps.


\begin{algorithm2e}[t]\label{alg:ARC_Random}{\caption {First-Order Stochastic Adaptive Regularization with Cubics (First-Order SARC) }}
		\SetAlgoLined
		\begin{description}
			
			\item[{Input:}]\ Oracles $\mathsf{SZO}(\epsilon_f, \lambda,a)$, $\mathsf{SFO}(\kappa_g)$ and  $\mathsf{SSO}(\kappa_H)$; initial iterate $x_0$, parameters $\gamma \in (0,1)$, $\theta\in (0,1)$, $\delta_1,\delta_2\in [0,\frac 1 2)$, $\sigma_{\min}>0$, $\eta\in (0,1)$, $\mu \geq 0,\epsilon_f^\prime>0$ and $\sigma_0\geq\sigma_{\min}$. 
				
			\item[{Repeat for $k=0, 1, \ldots$}]\ 
				
			\item[{\quad  1. Compute a model trial step $s_k$:}]\ 
			Generate $g_{k}=g(x_k,\xi_{k}^1)$ and $H_{k}=H(x_k,\xi_{k}^2)$ using $\mathsf{SFO}(\kappa_g)$ and $\mathsf{SSO}(\kappa_H)$ with $(\frac{\mu}{\sigma_k},\delta_1)$, and $(\sqrt{\frac{\mu}{\sigma_k}},\delta_2)$ as inputs, respectively.
				Compute  a trial step $s_k$ that satisfies \eqref{s-calc} and \eqref{TCs}  with parameters $\eta$ and $\sigma_k$ at $x_k$.
			\item[{\quad 2. Check sufficient  decrease:}]\ 
			Let $x_k^+=x_k+s_k$. 
			Compute function value estimations $f(x_k) = f(x_k,\xi_k)$ and  $f(x_k^+)=f(x_k^+,\xi_k^+)$ using the $\mathsf{SZO}(\epsilon_f, \lambda,a)$, and set
			\begin{equation}\label{rho}
				\rho_k=\frac{f(x_k)-f(x_k^+)+2\epsilon_f^\prime}{m(x_k)-m_k(x_k^+)},
			\end{equation} 	
			\item[{\quad 3. Update the iterate:}]\ 
			Set 
			\begin{equation}\label{newiterate} 
				x_{k+1}=\left\{ 
				\begin{array}{lr}
					x_k^+\quad {\rm if}\quad \rho_k\geq \theta & \text{[successful iteration]}\\
					x_k, \quad {\rm otherwise} & \text{[unsuccessful iteration]}
				\end{array}
				\right.
			\end{equation}
			\item[{\quad 4. Update the regularization parameter $\sigma_k$:}]\ 
			Set 
			\begin{equation}\label{regularization-update} 
				\sigma_{k+1}=\left\{ 
				\begin{array}{ll}
					\max\left\{\gamma\sigma_k,\sigma_{\min}\right\}
					\quad {\rm if}\quad \rho_k\geq \theta  \\
					\frac{1}{\gamma}\sigma_k,   \quad {\rm otherwise.}
				\end{array}
				\right.
			\end{equation}
		\end{description}
	\end{algorithm2e}
	
	\subsection{Deterministic Properties of Algorithm 1}
\label{sec:det_props}
	
 	Algorithm   \ref{alg:ARC_Random}   generates a  stochastic process and we will analyze it in the next section. In the algorithm, $\epsilon_f^{\prime}$ is an input parameter, which is only required to be some upper bound for the term ${\mathbb E_{\Xi(x)}}\left [ |{\phi(x) - f(x, \Xi(x))}|\right]$, not necessarily the tightest one. We have the  assumption $\epsilon_f \leq \epsilon_f^{\prime}$ as in \Cref{epsilonf}. We will later introduce the relationship between the convergence neighborhood $\epsilon$ and $\mu$, $\epsilon_f^{\prime}$, $\delta_1, \delta_2$ and other algorithm parameters.
 
	First, we state and prove several lemmas that establish the behavior of the algorithm {\em for every realization}. 
	A key concept that will be used in the analysis is the concept of a {\em true iteration}. Let $e_k=| f(x_k)-\phi(x_k)|$ and 
		$e_k^+= |f(x_k^+)-\phi(x_k^+)|$. 
	\begin{definition}[True iteration]
		\label{def:true_noisy}
		We say that iteration $k$ is \textbf{true} if   
			\begin{align}
			\|\nabla \phi(x_k)-g_k\|\leq \kappa_g & \max\left\{ \frac {\mu}{\sigma_k},\|s_k\|^2\right\} \text{, }
			\|(\nabla^2 \phi(x_k)-H_k)s_k\|\leq \kappa_H \max\left\{ \frac {\mu}{\sigma_k},\|s_k\|^2\right\} \label{eq:true1} \\
		& \text{ and }
			e_k+e_k^+ \leq 2\epsilon_f^\prime.
		\end{align}
\end{definition}
\begin{remark}	We will show in \Cref{algo1_imply_suff_acc} that by using $\mathsf{SFO}$ and $\mathsf{SSO}$ with the respective inputs, $\mu_1=\frac {\mu}{\sigma_k}$ in \eqref{sfo} and $\mu_2=\sqrt{\frac {\mu}{\sigma_k}}$ in \eqref{sso}, each iteration of Algorithm \ref{alg:ARC_Random} satifies \eqref{eq:true1} with probability at least $1-\delta_1-\delta_2$.   However, we note that the conditions in \eqref{eq:true1} can be satisfied by using more relaxed inputs (compared to those in the algorithm), where $\mu_1=\max \{\frac {\mu}{\sigma_k}, \|s_k\|^2\}$ in \eqref{sfo}, and $\mu_2=\max \{\frac {\mu}{\sigma_k\|s_k\|}, \|s_k\|\}$ in \eqref{sso}, instead of $\mu_1=\frac {\mu}{\sigma_k}$ and $\mu_2=\sqrt{\frac {\mu}{\sigma_k}}$. Since $s_k$ depends on the output of the oracles, implementing these more relaxed parameters is non-trivial and may require modifications to Algorithm \ref{alg:ARC_Random}. We leave it as a subject of future research.  
 \end{remark}
 
{We will now prove a sequence of results that hold for each realization of Algorithm \ref{alg:ARC_Random}, and are essential for the complexity analysis. 
 The two key results are \Cref{cor1} and \Cref{p-arc:steplength}, where \Cref{cor1} shows that until an $\epsilon$-stationary point is reached, every true iteration with large enough $\sigma_k$ is successful, and \Cref{p-arc:steplength} establishes the lower bound on function improvement on true and successful iterations. \Cref{p-arc:model_decrease,p-arc:sigmamax,lem:gainontrue,p-arc:barepsilon} lay the building blocks for them:
 On every successful iteration, the function improvement is lower bounded in  terms of the norm of the step (Lemma \ref{p-arc:model_decrease}).  There is a threshold value of  $\sigma_k$  where any true iteration  is either always successful or results in a very small step (\Cref{p-arc:sigmamax}). 
  When an iteration is true and the step is not very small, the norm of the step is lower bounded in terms of $\|\nabla \phi(x_k^+)\|$ (\Cref{lem:gainontrue}). Until an $\epsilon$-stationary point is reached, the step cannot be too small on true iterations (\Cref{p-arc:barepsilon}).}

		\begin{lemma}[Improvement on successful iterations]\label{p-arc:model_decrease}
		Consider any realization of Algorithm \ref{alg:ARC_Random}. For each iteration $k$, we have
		\begin{equation}\label{cr:modeldecrease}
			m_k(x_k)-m_k(x_k^+)\geq \frac{1}{6}\sigma_k\|s_k\|^3.
		\end{equation}
		On every successful iteration $k$, we have
		\begin{equation}\label{cr:fctdecrease}
			f(x_k)-f(x_{k+1})\geq \frac{\theta}{6}\sigma_k\|s_k\|^3 -2\epsilon_f^\prime,
		\end{equation} which implies
		\begin{equation}\label{cr:fctdecrease1}
			\phi(x_k)-\phi(x_{k+1})\geq \frac{\theta}{6}\sigma_k\|s_k\|^3 -e_k-e_k^+ -2\epsilon_f^\prime.
		\end{equation} 
		
	\end{lemma}
	\begin{proof}
	The proof is similar to the proof of Lemma 3.3 in \cite{cagotoPI}.
	Clearly, \eqref{cr:fctdecrease}
		follows from \eqref{cr:modeldecrease} and the sufficient decrease condition \eqref{rho}-\eqref{newiterate}: $$\frac{f(x_k)-f(x_k^+)+2\epsilon_f^\prime}{m_k(x_k)-m_k(x_k^+)}\geq \theta, $$ 
		and  \eqref{cr:fctdecrease1} follows from the definition of $e_k$ and $e_k^+$.
		
		It remains to prove \eqref{cr:modeldecrease}.
		Combining the first condition on step $s_k$ in \eqref{s-calc}, with the model expression for $s=s_k$,
		we can write 
		\[
		m_k(x_k)-m_k(x_k^+)=\frac{1}{2}(s_k)^TH_ks_k+\frac{2}{3}\sigma_k\|s_k\|^3.
		\]
		The second condition on $s_k$ in \eqref{s-calc} implies
		$(s_k)^TH_ks_k\geq -\sigma_k\|s_k\|^3$. Together with the
		above equation, we obtain \eqref{cr:modeldecrease}.
  
\end{proof}

	\begin{lemma}[Large $\sigma_k$ guarantees success or small step] \label{p-arc:sigmamax}
		Let Assumption \ref{ass:Lip_nablaf} hold. 
		For any realization of  Algorithm \ref{alg:ARC_Random}, if iteration
		$k$ is true,  and if
		\begin{equation}\label{sigmamax}
			\sigma_k\geq  \bar \sigma= \frac{2\kappa_g+\kappa_H +L+L_H}{1-\frac{1}{3}\theta},
		\end{equation}
		then iteration $k$ is either successful or produces $s_k$ such that $\|s_k\|^2< \frac{\mu}{\sigma_k}$. 
	\end{lemma}
	
	\begin{proof}
		
		Clearly, if $\rho_k-1\geq 0$, then $k$ is successful by definition. Let us consider 
		the case when $\rho_k<1$; then if $1-\rho_k\leq 1-\theta$, $k$ is
		successful. We have from \eqref{rho}, that 
		\[
		1-\rho_k=\frac{m_k(x_k)-m_k(x_k^+)- f(x_k)+f(x_k^+)-2\epsilon_f^\prime}{m_k(x_k)-m_k(x_k^+)}.
		\]
		Notice that: 
				\begin{align*}
			&m_k(x_k)-m_k(x_k^+)- f(x_k)+f(x_k^+)-2\epsilon_f^
			\prime 
			= f(x_k^+)-\left(f(x_k)+ s_k^Tg_k+\frac{1}{2}s_k^TH_ks_k+\frac{\sigma_k}{3}\|s_k\|^3\right)-2\epsilon_f^
			\prime \\
			&\leq
			\phi(x_k^+)-\left(\phi(x_k)+ s_k^Tg_k+\frac{1}{2}s_k^TH_ks_k+\frac{\sigma_k}{3}\|s_k\|^3\right)-2\epsilon_f^
			\prime +e_k +e_k^+\leq
			\phi(x_k^+)-\phi(x_k)- s_k^Tg_k-\frac{1}{2}s_k^TH_ks_k-\frac{\sigma_k}{3}\|s_k\|^3.
		\end{align*}
		The second to last inequality follows from the definition of $e_k$ and $e_k^+$, and the last inequality due to the iteration being true.
		
		Taylor expansion and Cauchy-Schwarz inequalities give, for some $\tau\in [x_k,x_k^+]$,
		\[
		\begin{array}{l}
			\phi(x_k^+)-\phi(x_k)- s_k^Tg_k-\frac{1}{2}s_k^TH_ks_k-\frac{\sigma_k}{3}\|s_k\|^3\\[1ex] 
			=[\nabla \phi(x_k)-g_k]^Ts_k
			+\frac{1}{2}(s_k)^T[\nabla ^2\phi(\tau)-\nabla^2\phi(x_k)]s_k +
			\frac{1}{2}(s_k)^T[\nabla ^2\phi(x_k)-H_k]s_k-\frac{1}{3}\sigma_k\|s_k\|^3\\[1ex]
			\leq \|\nabla \phi(x_k)-g_k\|\cdot\|s_k\|
			+\frac{1}{2} \|\nabla ^2\phi(\tau)-\nabla ^2\phi(x_k)\|\cdot \|s_k\|^2 +
			\frac{1}{2}\|(\nabla ^2\phi(x_k)-H_k)s_k\|\cdot\|s_k\|-\frac{1}{3}\sigma_k\|s_k\|^3\\[1ex]
			\leq(\kappa_g+ \frac{\kappa_H}{2})\max \left \{\frac{\mu}{\sigma_k}, \|s_k\|^2\right\} \|s_k\|
			+\left( \frac{L_H}{2}-\frac{1}{3}\sigma_k\right)\|s_k\|^3
		\end{array}
		\]
		where the last inequality follows from the fact that 
		the iteration is true and hence \eqref{eq:true1} holds: 
		$
		\|\nabla \phi(x_k)-g_k\|\leq \kappa_g\max\left\{ \frac {\mu}{\sigma_k},\|s_k\|^2\right\}\quad {\rm and}\quad
		\|(\nabla^2 \phi(x_k)-H_k)s_k\|\leq \kappa_H \max\left\{\frac{\mu}{\sigma_k},\|s_k\|^2\right \}$ and from 
		Assumption \ref{ass:Lip_nablaf}. So as long as $\|s_k\|^2\geq \frac{\mu}{\sigma_k}$, we have
		\[
		m_k(x_k)-m_k(x_k^+)- f(x_k)+f(x_k^+)-2\epsilon_f^\prime\leq 
		\left(\kappa_g+\frac{\kappa_H}{2}
		+\frac{L_H}{2}-\frac{1}{3}\sigma_k\right)\|s_k\|^3
		= (6\kappa_g + 3L_H+3\kappa_H -2\sigma_k)\frac{1}{6}\|s_k\|^3,
		\]
		which, together with \eqref{cr:modeldecrease}, gives that
		$1-\rho_k\leq 1-\theta$ when $\sigma_k$ satisfies \eqref{sigmamax}.
		
	\end{proof}
	
	Note that for the above lemma to hold, $\bar \sigma$ does not need to include $L$ in the numerator. However, we will need another condition on $\bar \sigma$ later that will involve $L$; hence for simplicity of notation we introduced $\bar \sigma$ above to satisfy all necessary bounds.
	

	\begin{lemma}[Lower bound on step norm in terms of $\|\nabla \phi(x_k^+)\|$]\label{lem:gainontrue}
		Let Assumption \ref{ass:Lip_nablaf} hold. 
		For any realization of Algorithm \ref{alg:ARC_Random},
		if $k$ is a true iteration  we have
		\begin{equation}\label{step-length-arc}
			\max\left\{\|s_k\|^2,\frac{\mu}{\sigma_k}\right\} \geq {\frac{1-\eta}{\sigma_k+(1-\frac{\theta}{3})\bar \sigma}\|\nabla \phi(x_k^+)\|}.
		\end{equation}
	\end{lemma}
	\begin{proof}
		
		The triangle inequality, the equality $\nabla
		m_k(x_k+s)=g_k+H_ks+\sigma_k\|s\| s$ and condition \eqref{TCs} on $s_k$ together give
		\begin{equation}\label{eq:gradbound}
			\begin{array}{lcl}
				\|\nabla \phi(x_k^+)\|&\leq& \|\nabla \phi(x_k^+)-\nabla m_k(x_k^+)\| +
				\|\nabla m_k(x_k^+)\|\\
				&\leq &\|\nabla \phi(x_k^+)-g_k-H_ks_k\|+\sigma_k\|s_k\|^2 +\eta\min\left\{1,\|s_k\|\right\}\|g_k\|.
			\end{array}
		\end{equation}
		Recalling Taylor expansion of $\nabla \phi(x_k^+)$:
		$
		\nabla \phi(x_k^+) =\nabla \phi(x_k)+\int_0^1 \nabla ^2\phi(x_k+ts_k)s_kdt,
		$
		and applying triangle inequality again, we have
		\[
		\begin{array}{lcl}
			\|\nabla \phi(x_k^+)-g_k-H_ks_k\|&\leq& \|\nabla \phi(x_k)-g_k\|+\left\|\int_0^1
			[\nabla ^2\phi(x_k+ts_k)-\nabla ^2\phi(x_k)]s_kdt\right\| + \|\nabla ^2\phi(x_k)s_k-H_ks_k\| \\[1ex]
		&\leq &(\kappa_g+\kappa_H)\max\left\{ \frac{\mu}{\sigma_k}, \|s_k\|^2\right \}+ \frac{1}{2}L_H \|s_k\|^2,
		\end{array}
		\]
		where to get the second inequality, we also used \eqref{eq:true1}
		and Assumption \ref{ass:Lip_nablaf}.
		We can bound $\|g_k\|$ as follows:
		\[
		\|g_k\|\leq \|g_k-\nabla \phi(x_k)\|+\|\nabla \phi(x_k)-\nabla \phi(x_k^+)\|
		+\|\nabla \phi (x_k^+)\|\leq \kappa_g\max\left\{ \frac {\mu}{\sigma_k},\|s_k\|^2\right\}+L\|s_k\|+\|\nabla \phi(x_k^+)\|.
		\]
		Thus finally, we can bound all the terms on the right hand side of \eqref{eq:gradbound} in terms of $\|s_k\|^2$ and using the fact that $\eta\in (0,1)$ we can write 
		\begin{align*}
			(1-\eta)\|\nabla \phi(x_k^+)\|&\leq (2\kappa_g +\kappa_H)\max\left\{ \frac{\mu}{\sigma_k}, \|s_k\|^2\right \}
			+(L+L_H+\sigma_k)\|s_k\|^2\\
			&\leq (2\kappa_g+\kappa_H+L+L_H+\sigma_k)
			\max \left\{\frac{\mu}{\sigma_k}, \|s_k\|^2\right\},
		\end{align*}
		
		which is equivalent to \eqref{step-length-arc}.
	\end{proof}

	
	\begin{lemma}[Lower bound on step norm  until $\epsilon$-accuracy is reached]\label{p-arc:barepsilon} Let Assumption \ref{ass:Lip_nablaf} hold. 
		Consider any realization of  Algorithm \ref{alg:ARC_Random}. 
	Let $\epsilon$ satisfy 
		\begin{equation}\label{eq:bar_epsilon}
			{\mu}\leq  \frac{1-\eta}{1+\frac{(1-\frac{\theta}{3})\bar\sigma}{\sigma_{\min}}}\epsilon.
		\end{equation}
		Then on each true iteration $k$  such that $\|\nabla \phi(x_k^+)\| \geq \epsilon$,  we  have
						\[
		\|s_k\|^2\geq \frac{\mu}{\sigma_k}.
		\]
	\end{lemma}
	\begin{proof}
		
		If iteration $k$ is true and $\|\nabla \phi(x_k^+)\|> \epsilon$, then by Lemma \ref{lem:gainontrue}:
		\begin{equation*}
			\max\left\{\|s_k\|^2,\frac{\mu}{\sigma_k}\right\} \geq {\frac{1-\eta}{\sigma_k+(1-\frac{\theta}{3})\bar\sigma}\|\nabla \phi(x_k^+)\|}
			> \frac{1-\eta}{\sigma_k+(1-\frac{\theta}{3})\bar\sigma}\epsilon,
		\end{equation*}
		but since  
		\[
		{\mu}\leq  \frac{1-\eta}{1+\frac{(1-\frac{\theta}{3})\bar\sigma}{\sigma_{\min}}}\epsilon,
		\]
		so 
		\[
		\frac{\mu}{\sigma_k}
		\leq  \frac{1-\eta}{\sigma_k+\frac{(1-\frac{\theta}{3})\bar\sigma\sigma_k}{\sigma_{\min}}}\epsilon
		\leq
		\frac{1-\eta}{\sigma_k+(1-\frac{\theta}{3})\bar\sigma}\epsilon.
		\]
		Hence, we must have
		\[
		\|s_k\|^2> \frac{1-\eta}{\sigma_k+(1-\frac{\theta}{3})\bar\sigma}\epsilon \geq \frac{\mu}{\sigma_k}. 
		\]
	\end{proof}
	
	\begin{corollary}[True iteration  with large $\sigma_k$ must be successful]\label{cor1} Let Assumption \ref{ass:Lip_nablaf} hold. 
		Consider any realization of  Algorithm \ref{alg:ARC_Random}. 
		Let $\epsilon$ satisfy \eqref{eq:bar_epsilon} and
		\begin{equation*}
			\sigma_k\geq \bar \sigma= \frac{2\kappa_g+\kappa_H +L+L_H}{1-\frac{1}{3}\theta},
		\end{equation*}
		then if iteration
		$k$ is true and   $\|\nabla \phi(x_k^+)\|> \epsilon$, then iteration $k$ is successful. 
	\end{corollary}
	\begin{proof}
		
		The result is straightforward by applying Lemma \ref{p-arc:sigmamax} and \ref{p-arc:barepsilon}.
	\end{proof}
	
	
	\begin{lemma}[Minimum improvement achieved by true and successful iterations]\label{p-arc:steplength} 
	Let Assumption \ref{ass:Lip_nablaf} hold. 
		Consider any realization of  Algorithm \ref{alg:ARC_Random}. 
		Let $\epsilon$ satisfy \eqref{eq:bar_epsilon}.  Then
		on each true and successful iteration $k$ for which $\norm{\nabla \phi(x_{k+1})}> \epsilon$, we have
		\begin{align}\label{cr:fctdecrease2}
			\phi(x_k)-\phi(x_{k+1}) 
			& \geq \frac{\theta}{6}(1-\eta)^{3/2}\frac{\sigma_k}{(\sigma_k+(1-\frac{\theta}{3})\bar\sigma)^{3/2}}\|\nabla \phi(x_{k+1})\|^{3/2} -e_k-e_k^+ -2\epsilon_f^\prime \\
			& \geq \frac{\theta}{6}(1-\eta)^{3/2}\frac{\sigma_{\min}}{(\sigma_k+(1-\frac{\theta}{3})\bar\sigma)^{3/2}}\|\nabla \phi(x_{k+1})\|^{3/2} -e_k-e_k^+ -2\epsilon_f^\prime, 
		\end{align}
		where $\bar\sigma$ is defined in \eqref{sigmamax}. 
	\end{lemma}
	\begin{proof}
		
		Combining Lemma \ref{lem:gainontrue}, \ref{p-arc:barepsilon},    inequality  \eqref{cr:fctdecrease1} from Lemma \ref{p-arc:model_decrease} and the definition of successful iteration in  Algorithm \ref{alg:ARC_Random}  we have, for all  true and successful iterations $k$,
		\begin{align}\label{cr:fctdecrease4}
			\phi(x_k)-\phi(x_{k+1})
			&\geq \frac{\theta}{6}\sigma_k\|s_k\|^3 -e_k-e_k^+ -2\epsilon_f^\prime\\
			&\geq  \frac{\theta}{6}(1-\eta)^{3/2}\frac{\sigma_k}{(\sigma_k+(1-\frac{\theta}{3})\bar\sigma)^{3/2}}\|\nabla \phi(x_{k+1})\|^{3/2} -e_k-e_k^+ -2\epsilon_f^\prime.
		\end{align}
		
		Using the fact that $\sigma_k\geq \sigma_{\min}$, the result follows. 
	\end{proof}
	

We can now use these important  properties of Algorithm \ref{alg:ARC_Random} to show that stochastic process that the algorithm generates fits into the framework analyzed in \cite{jinhigh}. 
		
\subsection{Stochastic Properties of Algorithm 1} 
		
	Algorithm   \ref{alg:ARC_Random}   generates a stochastic process.
	 Let $M_{k}$ denote the collection of random variables
	$\left\{\Xi_{k}, \Xi_k^+,\Xi_{k}^{1},  \Xi_{k}^{2}\right\}$, 
	whose realizations are  
	$\left\{\xi_k, \xi_k^+, \xi_{k}^1, \xi_{k}^2 \right\}$. Let $\left\{{\cal F}_{k}: k\geq 0\right\}$
	 denote the filtration generated by  $M_{0}, M_{1}, \ldots, M_{k}$.  
		 At iteration $k$, $X_k$ denotes the random iterate,  $G_k$ is the random gradient approximation, $\mathsf{H}_k$ is the random Hessian approximation,  $\Sigma_{k}$ is the random model regularization parameter.  $S_k$ is the step computed for the random model, $f(X_k,\Xi_k) \text{ and } f(X_k^+,\Xi_k^+)$ are the random function estimates at the current point and  the  candidate point, respectively. 

\sloppypar{Conditioned on $X_k$ and $\Sigma_k$, the random quantities ${G}_k$ and $\mathsf{H}_k$ are determined by  $\Xi_{k}^{1}$ and $\Xi_{k}^{2}$ respectively.  The realization of $S_k$ depends on  the realizations of ${G}_k$ and $\mathsf{H}_k$. The function estimates $f(X_k,\Xi_k) \text{ and } f(X_k^+,\Xi_k^+)$  are determined by $\Xi_k, \Xi_k^+$, conditioned on  $X_k$ and $X_k^+$. 
In summary, the stochastic process $\left\{\left({G}_k, \mathsf{H}_k, S_k, f(X_k,\Xi_k), f(X_k^+,\Xi_k^+), X_k, \Sigma_{k}\right)\right\}$ generated by the algorithm,  with realization 
$\left\{\left(g_k, H_k, s_k, f(x_k,\xi_k), f(x_k^+,\xi_k^+), x_k, \sigma_{k}\right)\right\}$, is adapted to the filtration $\left\{{\cal F}_{k}: k\geq 0\right\}$. }

	We further define $E_k:= |f(X_k,\Xi_k)-\phi(X_k)|$ and $E_k^+:= |f(X_k^+,\Xi_k^+)-\phi(X_k^+)|$,
	with realizations $e_k$ and $e_k^+$.  
	Let $\Theta_k := \one\left\{\text{iteration $k$ is successful}\right\}$, and let $I_k := \one\left\{\text{iteration $k$ is true}\right\}$. The indicator random variables $\Theta_k$ and $I_k$ are clearly measurable with respect to the filtration ${\cal F}_{k}$. 

		

%
	

	The next lemma shows that by construction of the algorithm, the stochastic model $m_k$ at iteration $k$ is ``sufficiently accurate" with probability at least $1-\delta_1-\delta_2$. 

	\begin{lemma}\label{algo1_imply_suff_acc}
	The indicator variable 
	\begin{align*}
		J_k =  \mathbbm{1} \bigg \{
		&	\|\nabla \phi(X_k)-g(X_k,\Xi_k^1(X_k))\|\leq  \kappa_g \max\left\{ \frac {\mu}{\Sigma_k},\|S_k\|^2\right\},
			 \text{ and }\\
		& \|(\nabla^2 \phi(X_k)-{H(X_k, \Xi_k^2(X_k))})S_k\|\leq  \kappa_H
		\max\left \{\frac{\mu}{\Sigma_k}, \|S_k\|^2\right \}
		\bigg \}
	\end{align*}

	satisfies the following condition
	$$\P(J_k=1\mid \mathcal{F}_{k-1})\geq 1-\delta_1-\delta_2.$$
	\end{lemma} 
 \begin{proof}
By the properties of oracles $\mathsf{SFO}$ and  $\mathsf{SSO}$ and the choices of the inputs for them in step 1 of the algorithm, we have:
		\begin{equation}\label{eq1}
			\P\left(\|\nabla \phi(X_k)-g(X_k,\Xi_k^1(X_k))\|\leq  \kappa_g \frac{\mu}{\Sigma_k}\mid  \mathcal{F}_{k-1} \right)\geq 1-\delta_1,
		\end{equation} and
		\begin{equation}\label{eq1.2}
			\P\left(\|(\nabla^2 \phi(X_k)-{H(X_k, \Xi_k^2(X_k))})\|\leq  \kappa_H\sqrt{\frac{\mu}{\Sigma_k}}\mid  \mathcal{F}_{k-1} \right)\geq 1-\delta_2.
		\end{equation}
		Inequality \eqref{eq1} implies 
		\begin{equation}\label{eq1.1}
		\P\left(\|\nabla \phi(X_k)-g(X_k,\Xi_k^1(X_k))\|\leq  \kappa_g \max\left\{ \frac {\mu}{\Sigma_k},\|S_k\|^2\right\}\mid  \mathcal{F}_{k-1} \right)\geq 1-\delta_1,
		\end{equation}
			and inequality \eqref{eq1.2} implies 
		\begin{equation}\label{eq2}
			\P\left(\|(\nabla^2 \phi(X_k)-{H(X_k, \Xi_k^2(X_k))})S_k\|\leq  \kappa_H
			\max\left\{\frac{\mu}{\Sigma_k}, \|S_k\|^2\right\}\mid  \mathcal{F}_{k-1} \right)\geq 1-\delta_2.
		\end{equation} 
	Thus, we conclude that
	$\P(J_k=1\mid \mathcal{F}_{k-1})\geq 1-\delta_1-\delta_2$ by the union bound.	
	\end{proof}

	Recall Definition \ref{def:true_noisy} and that we denote the event  of iteration $k$ being true by  indicator random variable $I_k$. It is crucial for our analysis that $\P(I_k=1\mid \mathcal{F}_{k-1})\geq p >\frac 1 2$ for all $k$. We will later combine Lemma \ref{algo1_imply_suff_acc} with the properties of  $\mathsf{SZO}$ for a bound on $\delta_1+ \delta_2$ to ensure $p >\frac 1 2$. 
       
	The iteration complexity  of our algorithm is defined as the following stopping time.
			\begin{definition}[First-order stopping time]		
		
		For $\epsilon > 0$,	$T_\epsilon:=\min\left\{k: \norm{\nabla \phi(X_{k}^+)} \leq \epsilon\right\} + 1$, the iteration complexity of the algorithm for reaching an  $\epsilon$-stationary point. We will refer to $T_\epsilon$ as the first-order \emph{stopping time} of the algorithm.
	\end{definition}
	It is important to note that even if for some iteration $k$, $\norm{\nabla\phi(X_{k}^+)} \leq \epsilon$, this iteration may not be successful and thus $\norm{\nabla\phi(X_{k+1})}$ may be greater than $\epsilon$. This is a consequence of the complexity analysis of cubic regularization methods that measure progress in terms of the gradient at the trial point and not at the current iterate, and thus is not specific to SARC. The stopping time is thus defined as the first time at which the algorithm {\em computes} a point at which the gradient of $\phi$ is less than $\epsilon$. 
	
	It is easy to see that $T_\epsilon$  is a \emph{stopping time} of the stochastic process with respect to ${\cal F}_k$. 
	Given a level of accuracy $\epsilon$, we aim to derive a bound on the iterations complexity $T_\epsilon$ with overwhelmingly high probability.
	In particular, we will show the number of iterations until the stopping time $T_{\epsilon}$ is a sub-exponential random variable, whose value
	(with high probability) scales as ${O}(\epsilon^{-3/2})$, similarly to the deterministic case. 
	Towards that end, we define stochastic process $Z_k$ to measure the progress towards optimality.
	\begin{definition}[Measure of Progress]
		\label{def:progress} For each $k \geq 0$, let $Z_k \geq 0$ be a random variable measuring the progress of the algorithm at step $k$:
		$Z_k = \phi(X_k) - \phi^*,$ where $\phi^*$ is a {lower bound} of $\phi$.  
	\end{definition}

	Armed with these definitions, we will be able to state properties of the stochastic process generated by Algorithm \ref{alg:ARC_Random}, which lead to the desired bounds on $T_{\epsilon}$. These properties hold under certain conditions on the parameters used by Algorithm \ref{alg:ARC_Random}. We state these conditions here.

		\begin{assumption}\label{epsilonf}
		Define $u = \epsilon_f^\prime - \epsilon_f$ and $K=\kval$ as defined in \eqref{eq:zero_order2}, and $p = 1-\delta_1-\delta_2 -\exp\left(-\min\left\{\frac{u^2}{2K^2},\frac{u}{2K}\right\}\right)$.
		\begin{description}
		\item[(a)]  $\epsilon_f^\prime>\epsilon_f$,
		\item[(b)] $\delta_1+\delta_2$ are chosen sufficiently small so that $p>\frac{1}{2}$,
		\item[(c)]  
		\begin{equation}\label{eps_lower}
		\epsilon> \max\left\{ \frac{1+\frac{(1-\frac{\theta}{3})\bar\sigma}{\sigma_{\min}}}{1-\eta}\mu, 
		\frac{((2-\frac{\theta}{3})\bar\sigma)}{1-\eta} \left (\frac{24\epsilon_f^\prime}{(p-\frac 1 2)\theta \sigma_{\min}}\right)^{\frac{2}{3}} \right\}.
		\end{equation} 
\end{description}
\end{assumption}
\begin{remark}
\Cref{epsilonf} (c) gives a lower bound on the best accuracy the algorithm is guaranteed to achieve given the accuracy parameters related to the stochastic oracles $\mathsf{SFO/SSO}$ and $\mathsf{SZO}$. Specifically, $\epsilon_f^\prime$ is lower bounded by $\epsilon_f$, which is the ``irreducible" error of the zeroth-order oracle. We observe that, if $\epsilon_f^\prime\approx \epsilon_f$ then the term involving the error of the zeroth-order oracle in the lower bound of $\epsilon$ for SARC is ${O} (\epsilon_f^{\frac 2 3} )$,
which is better dependency than those of SASS in \cite{jinhigh} and the stochastic trust region algorithms in \cite{cao2022first}, where $\epsilon$ is lower bounded by  ${ O}\left(\sqrt{\epsilon_f}\right )$.

The dependence of $\epsilon$ on $\mu$ has a somewhat more complicated interpretation: $\mu$ can be chosen arbitrarily by the algorithm, as long as oracles $\mathsf{SFO/SSO}$ can deliver appropriate accuracy. Recall that in the algorithm, the accuracy input $\mu_1$ for $\mathsf{SFO}$ is $\frac{\mu}{\sigma_k}$ and the accuracy input $\mu_2$ for $\mathsf{SSO}$ is $\sqrt{\frac{\mu}{\sigma_k}}$. If $\sigma_k$ is bounded from above by a constant, then essentially $\epsilon$ is proportional to the best accuracy required of $\mathsf{SFO}$ during the algorithm procedure and it is proportional to the square of the best accuracy required of $\mathsf{SSO}$ during the algorithm procedure. 
This dependency is the same as in deterministic inexact algorithms as well as in \cite{jinhigh} and the stochastic trust region algorithms in \cite{cao2022first}. We will comment on the existence of the upper bound on $\sigma_k$ after our main complexity result. 
\end{remark}

The following theorem establishes key properties of the stochastic process generated by Algorithm \ref{alg:ARC_Random}, that are essential for the convergence analysis.  Similar properties are used  in \cite{jinhigh}  obtain high probability iteration complexity for a stochastic step search method. 

To remain consistent with the notation used in \cite{jinhigh}, we define the random variable $A_k:=\frac{1}{\Sigma_{k}}$, with realization $\alpha_k=\frac{1}{\sigma_{k}}$, and a constant  $\bar \alpha = \frac{1}{\bar\sigma}$. 

	\begin{theorem}\label{ass:alg_behave}
		Let Assumptions \ref{ass:Lip_nablaf} and \ref{epsilonf} hold. 
		For $\bar\alpha = \frac{1}{\bar \sigma}$ and the following non-decreasing function $h: \R\to\R$: 
		$$h(\alpha) = \frac{\theta}{6}(1-\eta)^{3/2}\frac{\sigma_{\min}}{(\frac{1}{\alpha}+\frac{(1-\frac{\theta}{3})}{\bar\alpha})^{3/2}}\epsilon^{3/2},$$
		the following hold for all $k<T_\epsilon-1$:
		\begin{itemize}
			\item[(i)] $\P(I_k = 1 \mid \sF_{k-1}) \geq p$ for all $k$. (Conditioning on the past, each iteration is true with probability at least $p$.)
			\item[(ii)] If $A_k  \leq  \bar \alpha$ and $I_k=1$  then $\Theta_k=1$. (True iterations with sufficiently small $\alpha_k$ are successful.) 
		\item[(iii)]   If  $I_k\Theta_k=1$ then $Z_{k+1}\leq Z_k-h(A_k)+{4\epsilon_f^\prime}$. (True, successful iterations make progress.)
			\item[(iv)] 	$h(\bar{\alpha})>{\frac{4\epsilon_f^\prime}{p-\frac 1 2}}$. (The lower bound of potential progress for an iteration with parameter $\bar \alpha$.)
					\item[(v)] $Z_{k+1}\leq Z_k+{2\epsilon_f^\prime+ {E}_k+{E}_k^+}$ for all $k$. (The ``damage'' at each iteration is bounded.)	       
	
		\end{itemize}
	\end{theorem}
\begin{proof}
	
	Part (i) follows from the assumptions on $p$ and the definition of the true iteration.	
	
	Part (ii) follows directly from Corollary \ref{cor1}. 
	
	Part (iii) follows  from Lemma \ref{p-arc:steplength}.
	
	Part (iv)  follows from the definitions of $\bar \alpha$, $h(\alpha)$, and inequality \eqref{eps_lower}. Specifically, plugging in the definitions of $\bar\alpha$ and $h(\cdot)$, one can show that the inequality $h(\bar{\alpha}) \geq {\frac{4\epsilon_f^\prime}{p-\frac 1 2}}$ is equivalent to $ \epsilon> 		\frac{((2-\frac{\theta}{3})\bar\sigma)}{1-\eta} \left (\frac{24\epsilon_f^\prime}{(p-\frac 1 2)\theta \sigma_{\min}}\right)^{\frac{2}{3}}$, which holds by \Cref{epsilonf}.
	
	Part (v) has exactly the same proof as that of Proposition 1 part (v) in \cite{jinhigh} and is easily derived from the step acceptance condition of Algorithm \ref{alg:ARC_Random}.
\end{proof}

\subsection{High-Probability and Expected First-Order Iteration Complexity}
In \cite{jinhigh}, the authors derive a high probability bound on $T_{\epsilon}$ for a stochastic process that satisfies properties analogous to those stated in Theorem \ref{ass:alg_behave}. By applying a similar analysis to our setting, we can establish the first-order iteration complexity theorem for Algorithm \ref{alg:ARC_Random}.
One key modification however is that in our analysis, we bound the event $T_{\epsilon}\leq t+1$ rather than $T_{\epsilon}\leq t$. This difference arises from our distinct definition of the stopping time compared to \cite{jinhigh}.


	\begin{theorem}
		\label{thm:main_full}
		\sloppypar{
		Suppose Assumptions  \ref{ass:Lip_nablaf} and
		\ref{epsilonf} hold for Algorithm \ref{alg:ARC_Random}, then we have the following bound on the iteration complexity: 
		 for any $s \geq 0$, $\hat{p} \in \left( \frac12 + \frac{4\epsilon_f^\prime+s}{c_1\epsilon^{3/2}}, p\right)$, and 
		$t \geq \frac{R}{\hat{p} - \frac12 - \frac{4\epsilon_f^\prime+s}{c_1\epsilon^{3/2}}}$, we have
		$$\P\left(T_\varepsilon \leq t+1\right) \geq 1 - \exp\left(-\frac{(p-\hat{p})^2}{2p^2}t\right) - \exp\left(-\min\left\{\frac{s^2t}{{8K^2}},\frac{st}{{4K}}\right\}\right),$$
		where $c_1=\frac{\theta}{6}(1-\eta)^{3/2}\frac{\sigma_{\min}}{((2-\frac{\theta}{3})\bar\sigma)^{3/2}}$, $K=\kval$, $R = \frac{\phi(x_0)-\phi^*}{c_1\epsilon^{3/2}}+\max\left\{-\frac{\ln \alpha_0 + \ln \bar\sigma}{2\ln \gamma}, 0\right\}$, with $p$ and $\bar\sigma$ as defined previously.}
	\end{theorem}	
	
	\begin{remark}
		\begin{enumerate}

			\item 
			Theorem \ref{thm:main_full} shows the iteration complexity of Algorithm \ref{alg:ARC_Random} is $O(\epsilon^{-3/2})$ 
			with overwhelmingly high probability, which matches the deterministic counterpart. 
			

			\item The first-order SARC algorithm encounters an $\epsilon$-stationary point in a finite number of iterations with probability $1$. This is a direct consequence of the Borel–Cantelli lemma.
			\item Since the probabilities of the failure events $\{T_{\epsilon} > t+1\}$ are exponentially decaying for all $t\geq \Theta(\epsilon^{-3/2})$, this implies a complexity bound of $O(\epsilon^{-3/2})$ in expectation for the first-order SARC, i.e.,	
			$$\E[T_\epsilon] = O(\epsilon^{-3/2}).$$
			
		\end{enumerate}
	\end{remark}

\section{Second-Order SARC}\label{sec:second_order}
 In this section, we modify the first-order SARC algorithm and extend the analysis to obtain a second-order SARC algorithm and show that it achieves an approximate second-order stationary point in $O(\epsilon^{-3/2})$ iterations with overwhelmingly high probability. 

Similar to the first-order stopping time, in \Cref{stopping2} we define a second-order stopping time for the second-order SARC algorithm as the first time at which the algorithm {\em computes} a point $x$ at which both the gradient norm and the negative of the smallest eigenvalue of the Hessian are small enough, i.e., $\|\nabla \phi(x)\| \leq \epsilon$ and $\lambda_{\text{min}}(\nabla^2\phi(x)) \geq -\sqrt{\epsilon}$. 

\begin{algorithm2e}[t]\label{alg2:ARC_Random}{\caption {Second-Order Stochastic Adaptive Regularization with Cubics (Second-Order SARC) }}
	\SetAlgoLined
	\begin{description}
		
		\item[{Input:}]\ Oracles $\mathsf{SZO}(\epsilon_f, \lambda,a)$, $\mathsf{SFO}(\kappa_g)$ and  $\mathsf{SSO}(\kappa_H)$; initial iterate $x_0$, parameters $\gamma \in (0,1)$, $\theta\in (0,1)$, $\delta_1,\delta_2\in [0,\frac 1 2)$, $\sigma_{\min}>0$, $\eta, {\eta_2} \in (0,1)$, $\mu \geq 0,\epsilon_f^\prime>0$ and $\sigma_0\geq\sigma_{\min}$. 
		
		\item[{Repeat for $k=0, 1, \ldots$}]\ 
		
		\item[{\quad  1. Compute a model trial step $s_k$:}]\ 
		Generate $g_{k}=g(x_k,\xi_{k}^1)$ and $H_{k}=H(x_k,\xi_{k}^2)$ using $\mathsf{SFO}(\kappa_g)$ and $\mathsf{SSO}(\kappa_H)$ with $\left(\min\left\{\frac {\mu}{\sigma_k}, \frac {\mu}{\sigma_k^2} \right\},\delta_1\right)$, and $\left(\min\left\{\sqrt{\frac {\mu}{\sigma_k}}, \sqrt{\frac {\mu}{\sigma_k^2}} \right\},\delta_2\right)$ as inputs, respectively.
		Compute  a trial step $s_k$ that satisfies \eqref{s-calc}, \eqref{TCs} and \eqref{new_cr} with parameters $\eta$, {$\eta_2$} and $\sigma_k$ at $x_k$.
		
		\item[{\quad Steps 2–4 }] are identical to Algorithm \ref{alg:ARC_Random} and omitted for brevity.
	\end{description}
\end{algorithm2e}

	\begin{definition}[Second-order stopping time]\label{stopping2}
	
	For $\epsilon > 0$, let 
	$$T_{\epsilon,\sqrt{\epsilon}}:=\min\left\{k: \norm{\nabla \phi(X_{k}^+)} \leq \epsilon \text{ and } \lambda_{\text{min}}(\nabla^2\phi(x)) \geq -\sqrt{\epsilon} \right\} + 1$$ be the iteration complexity of the algorithm for reaching an  $(\epsilon,\sqrt{\epsilon})$ second-order stationary point. We will refer to $T_{\epsilon,\sqrt{\epsilon}}$ as the \emph{second-order stopping time} of the algorithm.
\end{definition}

Before we introduce the second-order SARC algorithm, we will first introduce an additional criterion that the approximate minimizer of the cubic subproblem must satisfy in the second-order algorithm, which is essential for establishing the second-order analysis.
The following proposition, which appears as Proposition 1 in \cite{nesterov2006cubic}, provides a key property of exact solutions to the cubic subproblem.

\begin{proposition}\label{key_second}
	For any exact solution to the cubic subproblem, the following matrix inequality holds:
	$$H_k+\frac{1}{2}\sigma_{k}\norm{s_k}I \succeq 0.$$
\end{proposition}

This matrix inequality can be equivalently expressed in terms of the minimum eigenvalue:
\begin{equation*}
	\lambda_{\text{min}}(H_k)+\frac{1}{2}\sigma_{k}\norm{s_k} \geq 0.
\end{equation*}

Rearranging this inequality yields a lower bound on the step size:
\begin{equation}\label{eq_key}
	\norm{s_k} \geq \frac{-2\lambda_{\text{min}}(H_k)}{\sigma_{k}}.
\end{equation}

Any exact solution of the cubic subproblem satisfies \eqref{eq_key}. However, to allow the cubic subproblem to be solved approximately, we introduce a parameter $\eta_2\in(0,1]$ to define the new condition \eqref{new_cr} to be satisfied when solving the cubic subproblem approximately, which relaxes \eqref{eq_key}:
\begin{equation}\label{new_cr}
	\norm{s_k} \geq \eta_2\frac{-2\lambda_{\text{min}}(H_k)}{\sigma_{k}}, \text{ where } \eta_2 \in (0,1). 
\end{equation}
When $\eta_2=1$, this reduces to \eqref{eq_key}, corresponding to an exact solution. In second order SARC, the approximate solution of the cubic subproblem must satisfy both the conditions from first-order SARC and this additional condition \eqref{new_cr}. The parameter $\eta_2$ is a user-defined input - as $\eta_2$ approaches 1, each iteration can potentially make more progress, though finding an acceptable $s_k$ may take longer.

For solving the cubic subproblem approximately, Algorithm \ref{alg2:ARC_Random} can use methods from \cite{cagotoPI,carmon2019gradient} to minimize model \eqref{cubic}, or Nesterov's algorithm from \cite{nesterov2006cubic} to find a global minimizer.

The second-order SARC algorithm is described in Algorithm \ref{alg2:ARC_Random}. The key differences between the first-order and second-order SARC algorithms are:
\begin{enumerate}
	\item The second-order algorithm requires possibly higher accuracy from both the first- and second-order oracles.
	\item The cubic subproblem is solved with stricter termination conditions.
\end{enumerate}
As a result, all the lemmas and propositions established for the first-order SARC algorithm in Section \ref{sec:first_order} still hold for the second-order SARC algorithm.
We can define notions similarly as in the first-order setting.


\begin{definition}[True iteration for second-order SARC]
	\label{def:true_noisy_2} 
	We say that iteration $k$ is \textbf{true} for Algorithm \ref{alg2:ARC_Random} if   
	\begin{align}
		\|\nabla \phi(x_k)-g_k\|\leq \kappa_g & \max\left\{ \min\left\{\frac {\mu}{\sigma_k}, \frac {\mu}{\sigma_k^2} \right\},\|s_k\|^2\right\} \text{, }
		\|(\nabla^2 \phi(x_k)-H_k)s_k\|\leq \kappa_H \max\left\{ \min\left\{\frac {\mu}{\sigma_k}, \frac {\mu}{\sigma_k^2} \right\},\|s_k\|^2\right\} \label{eq:true_second} \\
		& \text{ and }
		e_k+e_k^+ \leq 2\epsilon_f^\prime.
	\end{align}
\end{definition}

Similar to the first-order analysis in Section \ref{sec:first_order}, we can show that the following lemma holds for the second-order setting.

\begin{lemma}\label{algo2_imply_suff_acc}
	The indicator variable 
	\begin{align*}
		J_k =  \mathbbm{1} \bigg \{
		&	\|\nabla \phi(X_k)-g(X_k,\Xi_k^1(X_k))\|\leq  \kappa_g \max\left\{\min\left\{ \frac {\mu}{\Sigma_k},  \frac {\mu}{\Sigma_k^2} \right\},\|S_k\|^2\right\},
		\text{ and }\\
		& \|(\nabla^2 \phi(X_k)-{H(X_k, \Xi_k^2(X_k))})S_k\|\leq  \kappa_H
		\max\left \{\min\left\{ \frac {\mu}{\Sigma_k},  \frac {\mu}{\Sigma_k^2} \right\}, \|S_k\|^2\right \}
		\bigg \}
\end{align*}
satisfies the following submartingale-like condition
$$\P(J_k=1\mid \mathcal{F}_{k-1})\geq 1-\delta_1-\delta_2.$$
\end{lemma} 

For the second-order analysis, the assumptions on the parameters are as follows.
\begin{assumption}\label{epsilonf2}
	Define $u = \epsilon_f^\prime - \epsilon_f$ and $K=\kval$ as defined in \eqref{eq:zero_order2}, and $p = 1-\delta_1-\delta_2 -\exp\left(-\min\left\{\frac{u^2}{2K^2},\frac{u}{2K}\right\}\right)$.
	\begin{description}
		\item[(a)]  $\epsilon_f^\prime>\epsilon_f$,
		\item[(b)] $\delta_1+\delta_2$ are chosen sufficiently small so that $p>\frac{1}{2}$,
		\item[(c)]  
		\begin{equation}\label{eps_lower2}
			\epsilon> \max\left\{ 
			\max\left\{\frac{1+\frac{(1-\frac{\theta}{3})\bar\sigma}{\sigma_{\min}}}{1-\eta}, \frac{4}{\sigma_{\min}^2},  \left( \frac 1 {{\eta_2}}+\frac{2L_H}{\bar\sigma}\right)^2 \right\}\mu, 
			\max\left\{\frac{2-\frac{\theta}{3}\bar\sigma}{1-\eta}, (\bar\sigma{/\eta_2}+2L_H)^2\right\} \left(\frac{24\epsilon_f^\prime}{(p-\frac 1 2)\theta \sigma_{\min}}\right)^{\frac{2}{3}}
			\right\}.
		\end{equation} 
	\end{description}
\end{assumption}

Similar to the first-order analysis, we want to show that properties analogous to those in \Cref{ass:alg_behave} hold. Before proving the theorem, we will first examine how the lemmas used in proving \Cref{ass:alg_behave} change in the second-order SARC algorithm.
By following a similar argument as in the first-order analysis, we can establish \Cref{p-arc2:sigmamax} for Algorithm \ref{alg2:ARC_Random}. This lemma shows that when the regularization parameter $\sigma_k$ is sufficiently large, either the iteration is successful or the step size is small.

\begin{lemma}[Large $\sigma_k$ guarantees success or small step]  \label{p-arc2:sigmamax}
	Let Assumption \ref{ass:Lip_nablaf}, and \ref{epsilonf2} hold. 
	For any realization of  Algorithm \ref{alg2:ARC_Random}, if iteration
	$k$ is true,  and if
	\begin{equation}
		\sigma_k\geq  \bar \sigma= \frac{2\kappa_g+\kappa_H +L+L_H}{1-\frac{1}{3}\theta},
	\end{equation}
	then iteration $k$ is either successful or produces $s_k$ such that $\|s_k\|^2< \min\left\{\frac{\mu}{\sigma_k}, \frac{\mu}{\sigma_k^2}
	\right\}$. 
\end{lemma}

We now obtain a lemma that provides a lower bound for $\norm{s_k}^2$ in the second-order SARC algorithm for all $k<T_{\epsilon, \sqrt{\epsilon}}-1$.

\begin{lemma}[Lower bound on step norm  until $(\epsilon,\sqrt{\epsilon})$-accuracy is reached]\label{norm_bound_second} Let Assumption \ref{ass:Lip_nablaf} and \ref{epsilonf2} hold. 
	Consider any realization of  Algorithm \ref{alg2:ARC_Random}, 
for any true iteration $k$, with $k<T_{\epsilon, \sqrt{\epsilon}}-1$, we have
\[
\|s_k\|^2\geq \min \left\{ \frac{1-\eta}{\sigma_k+\left(1-\frac{\theta}{3}\right) \bar{\sigma}}\epsilon, \frac{{\epsilon}}
{(\sigma_{k}{/\eta_2}+2L_H)^2}  \right\}.
\] 
Specifically, for any true iteration $k$ such that $\|\nabla \phi(x_k^+)\| \geq \epsilon$,  we  have
	\[
\|s_k\|^2\geq 
\frac{1-\eta}{\sigma_k+\left(1-\frac{\theta}{3}\right) \bar{\sigma}}\epsilon,
\]
and for any true iteration $k$  such that $\lambda_{\text{min}}(\nabla^2\phi(x_k^+))<-\sqrt{\epsilon}$,  we  have
\begin{equation}\label{part_2_of_h_fun}
	\norm{s_k}^2\geq 
	\frac{{\epsilon}}
	{(\sigma_{k}{/\eta_2}+2L_H)^2}.
\end{equation}
\end{lemma}
\begin{proof}
Suppose the iteration is true as in \Cref{def:true_noisy_2}, then:
\begin{align*}
	\lambda_{\text{min}}(\nabla^2\phi(x_k^+))
&\geq 
	\lambda_{\text{min}}(\nabla^2\phi(x_k)) -L_H\norm{s_k}\\
&\geq 	\lambda_{\text{min}}(H_k)-L_H\norm{s_k}-\min\left\{\sqrt{\frac {\mu}{\sigma_k}}, \sqrt{\frac {\mu}{\sigma_k^2}} \right\}\\
&\geq -\frac 1 {2{\eta_2}} \sigma_{k}\norm{s_k}-L_H\norm{s_k}-\min\left\{\sqrt{\frac {\mu}{\sigma_k}}, \sqrt{\frac {\mu}{\sigma_k^2}} \right\}.
\end{align*}
The first inequality comes from Hessian being Lipschitz, the second inequality is from the accuracy requirement of the second-order oracle, and the third inequality is by \Cref{new_cr}.

Hence, we have $-\lambda_{\text{min}}(\nabla^2\phi(x_k^+))\leq \frac 1 {2{\eta_2}} \sigma_{k}\norm{s_k}+L_H\norm{s_k}+\min\left\{\sqrt{\frac {\mu}{\sigma_k}}, \sqrt{\frac {\mu}{\sigma_k^2}} \right\}.$
Rearrange the above inequality, we get:
 $-\lambda_{\text{min}}(\nabla^2\phi(x_k^+))
 -\min\left\{\sqrt{\frac {\mu}{\sigma_k}}, \sqrt{\frac {\mu}{\sigma_k^2}} \right\}
 \leq 
 (\frac 1 {2{\eta_2}} \sigma_{k}
 +L_H)\norm{s_k},$ or
\begin{equation*}
 \norm{s_k}\geq \frac{-\lambda_{\text{min}}(\nabla^2\phi(x_k^+))
 	-\min\left\{\sqrt{\frac {\mu}{\sigma_k}}, \sqrt{\frac {\mu}{\sigma_k^2}} \right\}}{\frac 1 {2{\eta_2}} \sigma_{k}
 	+L_H}.
\end{equation*}
If the second-order stationary point has not been reached, then either $\nabla\phi(x_k^+)>\epsilon$ or  $\lambda_{\text{min}}(\nabla^2\phi(x_k^+))<-\sqrt{\epsilon}$. In the former case, the analysis falls back to the first-order analysis. In the later case we have:  $-\lambda_{\text{min}}(\nabla^2\phi(x_k^+))>\sqrt{\epsilon}$, and 
the above becomes:
\begin{equation*}
	\norm{s_k}\geq 
	\frac{\sqrt{\epsilon}
		-\min\left\{\sqrt{\frac {\mu}{\sigma_k}}, \sqrt{\frac {\mu}{\sigma_k^2}} \right\}}{\frac 1 {2{\eta_2}} \sigma_{k}
		+L_H}
	\geq 
	\frac{\sqrt{\epsilon}
		-\min\left\{\sqrt{\frac {\mu}{\sigma_{\min}}}, \sqrt{\frac {\mu}{\sigma_{\min}^2}} \right\}}{\frac 1 {2{\eta_2}} \sigma_{k}
		+L_H}
	\geq 
	\frac{\sqrt{\epsilon}
		-\sqrt{\frac {\mu}{\sigma_{\min}^2}}}{\frac 1 {2{\eta_2}} \sigma_{k}
		+L_H}
	.
\end{equation*}
By \Cref{epsilonf2}(c), we have
{
	\begin{equation}\label{add_eps_cond1}
		\epsilon\geq \frac{4\mu}{\sigma_{\min}^2}.
	\end{equation}
} We obtain $\sqrt{\frac {\mu}{\sigma_{\min}^2}}\leq \frac{\sqrt{\epsilon}}{2}$. Hence, we obtain the key inequality:
\begin{equation}\label{part_2_of_h_fun}
	\norm{s_k}\geq 
	\frac{{\sqrt{\epsilon}}}
	{\sigma_{k}{/\eta_2}+2L_H}.
\end{equation}
The result follows by squaring both sides.
\end{proof}

\begin{lemma}[Minimum improvement achieved by true and successful iterations for second-order SARC]\label{lemma9}
	Let \Cref{ass:Lip_nablaf} and \ref{epsilonf2} hold. Consider any realization of Algorithm \ref{alg2:ARC_Random}. For each iteration $k$ that is true and successful and $k<T_{\epsilon, \sqrt{\epsilon}}-1$, we have 	
	
		\begin{align*}
			\phi(x_k)-\phi(x_{k+1}) & \geq  \min\left\{
			\frac{\theta}{6}(1-\eta)^{3/2}\frac{\sigma_{k}}{(\sigma_{k}+{(1-\frac{\theta}{3})}\bar\sigma)^{3/2}},
			\frac{\theta}{6}\frac{{\sigma_{k}}}
			{(\frac{\sigma_k}{{\eta_2}}+2L_H)^3 }
			\right\}{\epsilon}^{3/2} -4\epsilon_f^\prime\\
			& \geq  \min\left\{
			\frac{\theta}{6}(1-\eta)^{3/2}\frac{\sigma_{\min}}{(\sigma_{k}+{(1-\frac{\theta}{3})}\bar\sigma)^{3/2}},
			\frac{\theta}{6}\frac{{\sigma_{\min}}}
			{(\frac{\sigma_k}{{\eta_2}}+2L_H)^3 }
			\right\}{\epsilon}^{3/2} -4\epsilon_f^\prime.
		\end{align*} 
	\end{lemma}

\begin{proof}

	By \Cref{norm_bound_second}, we have the following lower bound for $\norm{s_k}^2$ on any true and successful iteration $k$, where $k<T_{\epsilon, \sqrt{\epsilon}}-1$:
	
	\[
	\|s_k\|^2\geq 
	\min \left\{ \frac{1-\eta}{\sigma_k+\left(1-\frac{\theta}{3}\right) \bar{\sigma}}, 
	\frac{1}
	{(\sigma_{k}{/\eta_2}+2L_H)^2}  \right\}\epsilon.
	\]
	
	Combining the above inequality with \Cref{cr:fctdecrease1} in \Cref{p-arc:model_decrease}, together with the fact that iteration $k$ is true, we obtain the result.
	
\end{proof}

	\begin{lemma}[True iteration  with large $\sigma_k$ must be successful]\label{cor2} Let \Cref{ass:Lip_nablaf} and \ref{epsilonf2} hold. 
	Consider any realization of  Algorithm \ref{alg2:ARC_Random}. 
	Let 
	\begin{equation*}
		\sigma_k\geq \bar \sigma= \frac{2\kappa_g+\kappa_H +L+L_H}{1-\frac{1}{3}\theta},
	\end{equation*}
	then if iteration
	$k$ is true and $k<T_{\epsilon, \sqrt{\epsilon}}-1$, then iteration $k$ is successful. 
\end{lemma}
\begin{proof}

By \Cref{p-arc2:sigmamax}, we know that to complete the proof, it suffices to show that for all $k<T_{\epsilon, \sqrt{\epsilon}}-1$, 
\begin{align*}
    \|s_k\|^2 \geq \min\left\{\frac{\mu}{\sigma_k}, \frac{\mu}{\sigma_k^2}\right\}.
\end{align*}

By the first-order analysis, we have that if $\epsilon$ satisfy \Cref{epsilonf2} (and hence \Cref{epsilonf}), then for any true iteration such that  $\nabla\phi(x^+_k) \geq \epsilon$, we have $\norm{s_k}^2\geq \frac{\mu}{\sigma_{k}}\geq  \min\left\{\frac{\mu}{\sigma_k}, \frac{\mu}{\sigma_k^2}
\right\}$. On the other hand, by \Cref{norm_bound_second}, if the true iteration $k$ has $\lambda_{\text{min}}(\nabla^2\phi(X_{k}^+))\leq -\sqrt{\epsilon}$, we know 
$\norm{s_k}^2\geq 
\frac{{\epsilon}}
{(\sigma_{k}{/\eta_2}+2L_H)^2}.$
By \Cref{epsilonf2}(c), we have: 
{$\epsilon \geq \mu \left( \frac{1}{\eta_2}+\frac{2L_H}{\bar\sigma}\right)^2.
$}
As a result, for any $\sigma_k\geq \bar\sigma$, $\norm{s_k}^2\geq 
\frac{{\epsilon}}
{(\sigma_{k}{/\eta_2}+2L_H)^2}\geq \frac{\mu}{\sigma_{k}^2}\geq  \min\left\{\frac{\mu}{\sigma_k}, \frac{\mu}{\sigma_k^2}\right\}$.
This completes the proof. 
\end{proof}

Now, we are ready to state \Cref{ass2:alg_behave} for the second-order analysis which is analogous to \Cref{ass:alg_behave} for the first-order analysis. 
\begin{theorem}\label{ass2:alg_behave}
	Let Assumptions \ref{ass:Lip_nablaf} and \ref{epsilonf2} hold. 
	For $\bar\alpha = \frac{1}{\bar \sigma}$ and the following non-decreasing function $h: \R\to\R$: 
	$$	h(\alpha) = \min\left\{
	\frac{\theta}{6}(1-\eta)^{3/2}\frac{\sigma_{\min}}{(\frac{1}{\alpha}+\frac{(1-\frac{\theta}{3})}{\bar\alpha})^{3/2}}\epsilon^{3/2},
	\frac{\theta}{6}\frac{{\sigma_{\min}}}
	{(\frac{1}{\alpha{\eta_2}}+2L_H)^3 }{\epsilon}^{3/2}
	\right\},$$
	the following hold for all $k<T_{\epsilon,\sqrt{\epsilon}}-1$:
	\begin{itemize}
		\item[(i)] $\P(I_k = 1 \mid \sF_{k-1}) \geq p$ for all $k$. (Conditioning on the past, each iteration is true with probability at least $p$.)
		\item[(ii)] If $A_k  \leq  \bar \alpha$ and $I_k=1$  then $\Theta_k=1$. (True iterations with sufficiently small $\alpha_k$ are successful.) 
		\item[(iii)]   If  $I_k\Theta_k=1$ then $Z_{k+1}\leq Z_k-h(A_k)+{4\epsilon_f^\prime}$. (True, successful iterations make progress.)
		\item[(iv)] 	$h(\bar{\alpha})>{\frac{4\epsilon_f^\prime}{p-\frac 1 2}}$. (The lower bound of potential progress for an iteration with parameter $\bar \alpha$.)
		\item[(v)] $Z_{k+1}\leq Z_k+{2\epsilon_f^\prime+ {E}_k+{E}_k^+}$ for all $k$. (The ``damage'' at each iteration is bounded.)	       
		
	\end{itemize}
\end{theorem}

\begin{proof}
The proofs for parts (i) and (v) are the same as in the proof of \Cref{ass:alg_behave}. 

Part (ii) is a direct consequence of \Cref{cor2}.

By \Cref{lemma9}, we see that (iii) holds with 
\begin{equation}
	h(\alpha) = \min\left\{
	\frac{\theta}{6}(1-\eta)^{3/2}\frac{\sigma_{\min}}{(\frac{1}{\alpha}+\frac{(1-\frac{\theta}{3})}{\bar\alpha})^{3/2}}\epsilon^{3/2},
	\frac{\theta}{6}\frac{{\sigma_{\min}}}
	{(\frac{1}{\alpha{\eta_2}}+2L_H)^3 }{\epsilon}^{3/2}
	\right\}.
\end{equation} 

Part (iv) is a direct consequence of \Cref{epsilonf2}(c), together with the updated function of $h(\cdot)$.

\end{proof}
\subsection{High-Probability and Expected Second-Order Iteration Complexity}
Using an argument similar to the first-order analysis, we obtain the following main iteration complexity bound for the second-order SARC.
	
\begin{theorem}
	\label{thm2:main_full}
	\sloppypar{
		Suppose Assumptions  \ref{ass:Lip_nablaf} and
		\ref{epsilonf2} hold for Algorithm \ref{alg2:ARC_Random}, then we have the following bound on the iteration complexity: 
		for any $s \geq 0$, $\hat{p} \in \left( \frac12 + \frac{4\epsilon_f^\prime+s}{c_2\epsilon^{3/2}}, p\right)$, and 
		$t \geq \frac{R}{\hat{p} - \frac12 - \frac{4\epsilon_f^\prime+s}{c_2\epsilon^{3/2}}}$, we have
		$$\P\left(T_{\epsilon,\sqrt{\epsilon}} \leq t+1\right) \geq 1 - \exp\left(-\frac{(p-\hat{p})^2}{2p^2}t\right) - \exp\left(-\min\left\{\frac{s^2t}{{8K^2}},\frac{st}{{4K}}\right\}\right),$$
		where $c_2=\min\left\{
		\frac{\theta}{6}(1-\eta)^{3/2}\frac{\sigma_{\min}}{((2-\frac{\theta}{3})\bar\sigma)^{3/2}},
		\frac{\theta}{6}\frac{{\sigma_{\min}}}
		{(\bar\sigma{/\eta_2}+2L_H)^3 }
		\right\}$, $K=\kval$, $R = \frac{\phi(x_0)-\phi^*}{c_2\epsilon^{3/2}}+\max\left\{-\frac{\ln \alpha_0 + \ln \bar\sigma}{2\ln \gamma}, 0\right\}$, with $p$ and $\bar\sigma$ as defined previously.}
\end{theorem}	

\begin{remark}
	\begin{enumerate}

		\item 
		Theorem \ref{thm2:main_full} shows the iteration complexity of Algorithm \ref{alg2:ARC_Random} is $O(\epsilon^{-3/2})$ 
		with overwhelmingly high probability, which matches the deterministic counterpart. 
		

		\item The second-order SARC algorithm encounters an  $(\epsilon,\sqrt{\epsilon})$ second-order stationary point in a finite number of iterations with probability $1$. This is a direct consequence of the Borel–Cantelli lemma.
		\item Since the probabilities of the failure events $\{T_{\epsilon,\sqrt{\epsilon}} > t+1\}$ are exponentially decaying for all $t\geq \Theta(\epsilon^{-3/2})$, this implies a complexity bound of $O(\epsilon^{-3/2})$ in expectation for second-order SARC, i.e.
		$$
		\E[T_{\epsilon,\sqrt{\epsilon}}] = O(\epsilon^{-3/2}).
		$$
		
	\end{enumerate}
\end{remark}

\section{Sample Complexity}\label{sec:sample_complexity}
In addition to the iteration complexity, we also provide total sample complexity (or total oracle complexity) bounds for the algorithms. The total sample complexity of an algorithm is defined as the sum of the number of times the algorithm evaluates the function, the gradient, and the Hessian at some point.

Let's consider the problem of \textbf{Expectation minimization}:
$$
\min_{x\in \R^m} \phi(x)=\E_{d \sim \mathcal{D}}[l(x,d)].
$$ 
Here, $x \in \R^m$ denotes the model parameters, $d$ represents a data sample drawn from distribution $\mathcal{D}$, and $l(x,d)$ measures the loss incurred when the model with parameters $x$ is evaluated on data point $d$. This formulation encompasses supervised machine learning problems and appears in other domains such as simulation optimization \cite{kim2014guide}. We assume $\phi$ satisfies the assumptions of the paper, and that the function value $l(x,d)$, gradient $\nabla_x l(x, d)$ and Hessian $\nabla_x^2 l(x, d)$ can be computed for any $d \sim \mathcal{D}$. Moreover, we assume the following conditions hold:
\begin{itemize}
\item  $\ell(x, d)-\phi(x)$ is a 
subexponential random variable with parameters $\lambda$ and $a$ that satisfies \eqref{eq:zero_order}, and there is some $v_f \geq 0$ such that for all $x$, 
\begin{equation}\label{eq:erm_var_f}
\mathbb{E}_{d \sim \mathcal{D}}\left[\left(l(x, d)-\phi(x)\right)^2\right] \leq v_f^2;
\end{equation}
\item $\mathbb{E}_{d \sim \mathcal{D}}\left[\nabla_x \ell(x, d)\right]=\nabla \phi(x)$, and there is some $v_g \geq 0$ such that for all $x$,
\begin{equation}\label{eq:erm_var_g}
\mathbb{E}_{d \sim \mathcal{D}}\left\|\nabla_x \ell(x, d)-\nabla \phi(x)\right\|^2 \leq v_g^2.
\end{equation}
\item $\mathbb{E}_{d \sim \mathcal{D}}\left[\nabla_x^2 \ell(x, d)\right]=\nabla^2 \phi(x)$, and there is some $v_h \geq 0$ such that for all $x$,
\begin{equation}\label{eq:erm_var_h}
\mathbb{E}_{d \sim \mathcal{D}}\left\|\nabla_x^2 \ell(x, d)-\nabla^2 \phi(x)\right\|^2 \leq v_h^2.
\end{equation}
\end{itemize}


In this setting, the zeroth-, first- and second-order oracles are typically computed using minibatches:
\begin{equation}\label{eq:erm_oracle}
	f(x, \mathcal{S}_0) = \frac{1}{|\mathcal{S}_0|}\sum_{d\in \mathcal{S}_0}l(x,d),\quad g(x, \mathcal{S}_1) = \frac{1}{|\mathcal{S}_1|}\sum_{d\in \mathcal{S}_1}\nabla_x l(x,d), \quad H(x, \mathcal{S}_2) = \frac{1}{|\mathcal{S}_2|}\sum_{d\in \mathcal{S}_2}\nabla_x^2 l(x,d)
\end{equation}
where $\mathcal{S}_0, \mathcal{S}_1, \mathcal{S}_2$ represent minibatches - sets of i.i.d. samples drawn from distribution $\mathcal{D}$. The size of $\mathcal{S}_0, \mathcal{S}_1, \mathcal{S}_2$ may be chosen adaptively based on $x$.

To analyze the computational cost, we define some key metrics:
\begin{itemize}
    \item The \emph{total function value sample complexity} $\mathrm{TOC}_0$: The number of times the algorithm evaluates $l(x,d)$ for specific $x$ and $d$
    \item The \emph{total gradient sample complexity} $\mathrm{TOC}_1$: The number of times the algorithm computes $\nabla l(x,d)$
    \item The \emph{total Hessian sample complexity} $\mathrm{TOC}_2$: The number of times the algorithm computes $\nabla^2 l(x,d)$
\end{itemize}
The total sample complexity\footnote{The sample complexity considered here only accounts for the minibatch sizes used and does not depend on the dimension of $x$, which we treat as a constant.} (or total oracle complexity) of the algorithm in $n$ iterations is then defined as the sum of these quantities.
$$
\operatorname{TOC}(n)=\mathrm{TOC}_0(n)+\mathrm{TOC}_1(n)+\mathrm{TOC}_2(n).
$$



For both versions of the SARC algorithms, we observe that the oracle accuracy requirements change adaptively with the penalty parameter $\sigma_k$. It is clear that in the above expectation minimization setting, where a stochastic oracle is delivered via sample averaging, a higher accuracy requirement leads to higher oracle complexity. Therefore, to bound the total oracle complexity of the algorithm, we need to bound the accuracy requirement over the iterations, or equivalently, provide an upper bound for the parameter $\sigma_k$.

\medskip\noindent\textbf{Upper Bound on $\sigma_k$.}
While the penalty parameters $\Sigma_k$  form a stochastic process, this process has some nice properties. Specifically it is upper bounded by  a one-sided geometric  random walk. This random walk is analyzed in  \cite{jinhigh2} and it is shown that for any given number of iterations $t$, 
and for $\gamma$ chosen appropriately dependent on $t$, $\max_{1\leq k\leq t}\{\sigma_k\}\leq { O}(\bar \sigma)$ with high probability. 
A consequence of this fact is that with high probability, for a realization of Algorithm \ref{alg:ARC_Random} or \ref{alg2:ARC_Random}, there exist lower bounds on all of the accuracy requirements $\mu_1$ and $\mu_2$, which are inputs to the oracles $\mathsf{SFO}$ and $\mathsf{SSO}$ as in \eqref{sfo} and \eqref{sso}. 
This, in turn, can give total ``sample" complexity bounds for Algorithms \ref{alg:ARC_Random} and \ref{alg2:ARC_Random}. 

Theorem 3 from \cite{jinhigh2} is included here for completeness. It is easy to see that the theorem assumptions are satisfied by both versions of SARC algorithms and the expectation minimization problem setting.
The cost of an oracle call may depend on the step size parameter $\alpha_k=\frac{1}{\sigma_k}$ and the probability parameter $1-\delta_i$, thus we denote the cost of the zeroth-order oracle, first-order oracle and second-order oracle by $\mathrm{oc}_0(\alpha_k, 1-\delta_0)$, $\mathrm{oc}_1(\alpha_k, 1-\delta_1)$ and $\mathrm{oc}_2(\alpha_k, 1-\delta_2)$ respectively. 
We will use $\mathrm{oc}_0(\alpha_k)$, $\mathrm{oc}_1(\alpha_k)$ and $\mathrm{oc}_2(\alpha_k)$ in the paper to simplify the notation because without loss of generality, $\delta_i$ can be treated as constants in the algorithms. Let $\mathrm{oc}\left(\alpha_k\right)= \mathrm{oc}_0\left(\alpha_k\right)+
\mathrm{oc}_1\left(\alpha_k\right)+ \mathrm{oc}_2\left(\alpha_k\right)$.
Hence, 
$$\operatorname{TOC}_0(n)=\sum_{k=1}^n \mathrm{oc}_0\left({A}_k\right), \quad \operatorname{TOC}_1(n)=\sum_{k=1}^n \mathrm{oc}_1\left({A}_k\right), \quad  \operatorname{TOC}_2(n)=\sum_{k=1}^n \mathrm{oc}_2\left({A}_k\right),$$
$$\quad \text{and}\quad \operatorname{TOC}(n)=\sum_{k=1}^n \mathrm{oc}\left({A}_k\right).$$
Note that $\operatorname{TOC}(n)$ is a random variable since it depends on the random stepsize parameters $A_k$. The following theorem from \cite{jinhigh2} provides a bound on $\operatorname{TOC}(T_{\epsilon})$ by first analyzing $\alpha^*(n)$, the high probability minimum stepsize parameter in $n$ iterations, and then using $\mathrm{oc}\left(\alpha^*(n)\right)$ to bound the total sample complexity.

\begin{theorem}[Theorem 3 in \cite{jinhigh2}]\label{thm:toc_bound} For any $\omega>0$ and positive integer $n$, with probability at least $1-\mathbb{P}\left(T_{\varepsilon}>n\right)-n^{-\omega}-c n^{-(1+\omega)}$,

$$
\operatorname{TOC}\left(T_{\varepsilon}\right) \leq n \cdot \mathrm{oc}\left(\alpha^*(n)\right)
$$

where $\alpha^*(n)=\bar{\alpha} \gamma n^{-(1+\omega) \log _{1 / 2 q} 1 / \gamma}$, where $c=\frac{2 \sqrt{p q}}{(1-2 \sqrt{p q})^2}$ and $q=1-p$.
If $\gamma$ is chosen to be at least $\max \left\{\frac{1}{2},\left(\frac{1}{2 q}\right)^{\frac{\log (2 \beta)}{(1+\omega) \log n}}\right\}$ for some $\beta<\frac{1}{2}$, then $\alpha^*(n) \geq$ $\beta \bar{\alpha}$, thus with probability at least $1-\mathbb{P}\left(T_{\varepsilon}>n\right)-n^{-\omega}-c n^{-(1+\omega)}$,
$$
\mathrm{TOC}\left(T_{\varepsilon}\right) \leq n \cdot \mathrm{oc}(\beta \bar{\alpha}).
$$
\end{theorem}
For the remainder of this paper, we will use superscripts \( ^1 \) and \( ^2 \) to denote the oracle complexity related quantities for the first-order and second-order SARC algorithms, respectively.
\subsection{High Probability First-Order and Second-Order Sample Complexity Result}
The zeroth-, first- and second-order oracle conditions can be satisfied by choosing appropriate minibatch sizes for both SARC algorithms.

For the \textbf{first-order} SARC algorithm, by definition of the oracles $\mathsf{SZO}, \mathsf{SFO}$ and $\mathsf{SSO}$, the assumptions of the expectation minimization problem, as in \cref{eq:erm_var_f}, \cref{eq:erm_var_g}, \cref{eq:erm_var_h}, \Cref{epsilonf} and Chebyshev's inequality, we have:
$$\mathrm{oc}^1_1(\alpha_k) = \frac{v_g^2}{\delta_1 \kappa_{g}^2\mu^2 \alpha_k^2}= O\left(\frac{v_g^2}{\alpha_k^2\mu^2}\right)= O\left(\frac{v_g^2}{\alpha_k^2\epsilon^2}\right),$$ 
$$\mathrm{oc}^1_2(\alpha_k) =\frac{v_h^2}{\delta_2 \kappa_{h}^2\mu \alpha_k}= O\left(\frac{v_h^2}{\alpha_k\mu}\right)= O\left(\frac{v_h^2}{\alpha_k\epsilon}\right).$$
Using \Cref{epsilonf} and a similar argument as in Proposition 7 of \cite{jinhigh} shows that a sufficiently accurate zeroth-order oracle can be obtained for a nonconvex problem by using a minibatch of size:
$$\mathrm{oc}^1_0(\alpha_k)  = O\left(\frac{v_f^2}{\epsilon_f^2}\right) = O\left(\frac{v_f^2}{\epsilon^3}\right).$$ 

Similarly, for the \textbf{second-order} SARC algorithm, by \Cref{epsilonf2}, we have:
$$\mathrm{oc}^2_1(\alpha_k) = \frac{v_g^2}{\delta_1 \kappa_{g}^2\mu^2 \min\{\alpha_k^2, \alpha_k^4\}}= O\left(\frac{v_g^2}{\min\{\alpha_k^2, \alpha_k^4\}\mu^2}\right)= O\left(\frac{v_g^2}{\min\{\alpha_k^2, \alpha_k^4\}\epsilon^2}\right),$$ 
$$\mathrm{oc}^2_2(\alpha_k) = \frac{v_h^2}{\delta_2 \kappa_{h}^2\mu \min\{\alpha_k, \alpha_k^2\}}= O\left(\frac{v_h^2}{\min\{\alpha_k, \alpha_k^2\}\mu}\right)= O\left(\frac{v_h^2}{\min\{\alpha_k, \alpha_k^2\}\epsilon}\right),$$
$$\mathrm{oc}^2_0(\alpha_k) = O\left(\frac{v_f^2}{\epsilon_f^2}\right) = O\left(\frac{v_f^2}{\epsilon^3}\right).$$
We can now apply Theorem \ref{thm:toc_bound}, and conclude that with high probability, before the stopping time is reached, the smallest $\alpha_k$ above is at least $\beta \bar{\alpha}$, which is a constant. Hence, we can obtain the following result with high probability.


\begin{itemize}
	\item Total sample complexity for the zeroth-order oracle: 
	For both the first-order and second-order SARC algorithms, each iteration requires $O\left(\frac{v_f^2}{\eps^{3}}\right)$ samples for the function value estimates. Since there are $O\left(\frac{1}{\eps^{1.5}}\right)$ iterations in total, the overall sample complexity for the zeroth-order oracle is $O\left(\frac{v_f^2}{\eps^{4.5}}\right)$. 
	\item Total sample complexity for the first-order oracle: 
	For both the first-order and second-order SARC algorithms, each iteration requires $O\left(\frac{v_g^2}{\epsilon^2}\right)$ samples. Multiplying by the $O\left(\frac{1}{\eps^{1.5}}\right)$ total number of iterations, we obtain an overall sample complexity of $O\left(\frac{v_g^2}{\eps^{3.5}}\right)$.

	\item Total sample complexity for the second-order oracle: 
For both the first-order and second-order SARC algorithms, each iteration requires $O\left(\frac{v_h^2}{\epsilon}\right)$ samples. Therefore, the total sample complexity for the second-order oracle is $O\left(\frac{v_h^2}{\eps^{2.5}}\right)$.
\end{itemize}

Formally, we obtain the following high probability total sample complexity result for both the first-order and second-order SARC algorithms.
\begin{theorem}[High Probability Total Sample Complexity of both the first-order and second-order SARC] 
	\label{thm:sarc_hp}
	For algorithms \ref{alg:ARC_Random} and \ref{alg2:ARC_Random} applied to expectation minimization problem, let $n \geq R_1$ and $n \geq R_2$ for the first-order and second-order SARC algorithms respectively, where $R_1 =\frac{R}{\hat{p}-\frac{1}{2}-\frac{4 \epsilon_f^{\prime}+s}{c_1 \epsilon^{3 / 2}}}$ and $R_2 =\frac{R}{\hat{p}-\frac{1}{2}-\frac{4 \epsilon_f^{\prime}+s}{c_2 \epsilon^{3 / 2}}}$.
	For any $\omega  > 0$, $\gamma \in \left[\,\max\left\{\frac12, \left(\frac{1}{2q}\right)^{\frac{\log(2\beta)}{(1+\omega)\log n} } \right\}\,, 1\right)$, where $\beta$ is any constant smaller than $\frac12$, we have with probability at least $1-2\exp\left(-C_3n\right) - {\cal O}(n^{-\omega })$, where $C_3 =\min\left\{\frac{(p-\hat{p})^2}{2 p^2},\min \left\{\frac{s^2}{8 K^2}, \frac{s}{4 K}\right\}\right\}$, 
	\begin{equation}
		\SC(T_{\eps})  \leq {\cal O}\left(\left(\frac{1}{\eps^{\frac 3 2}}+\log_{1/\gamma} \frac{\alpha_0}{\bar{\alpha}}\right)\cdot\left( \frac{v_f^2}{\eps^3}  + \frac{v_g^2}{\eps^2} + \frac{v_h^2}{\eps} \right)\right).
	\end{equation}
	
	If $\gamma \in \left[\, \max\left\{\frac12, \left(\frac{1}{2q}\right)^{\frac{\log(2\beta)}{(1+\omega)\log n} } \right\}, \left(\frac{\bar{\alpha}}{\alpha_0}\right)^{\frac{1}{cn}}\,\right], $ where $c$ is any constant in $(0,1)$, the above simplifies to
	\begin{equation}
		\SC(T_{\eps})  \leq {\cal O}\left(\frac{1}{\eps^{\frac 3 2}}\cdot\left( \frac{v_f^2}{\eps^3}  + \frac{v_g^2}{\eps^2} + \frac{v_h^2}{\eps}\right)\right).
	\end{equation}
\end{theorem}

\begin{remark}
Comparing the sample complexity of the first-order and second-order SARC algorithms with the results in \cite{tripuraneni2018stochastic}:
\begin{itemize}
	\item The total sample complexity for the \emph{first-order} oracle in both the first-order and second-order SARC algorithms is $O\left(\frac{1}{\eps^{3.5}}\right)$ which matches the result of \cite{tripuraneni2018stochastic}.
	\item The total sample complexity for the \emph{second-order} oracle in both the first-order and second-order SARC algorithms is $O\left(\frac{1}{\eps^{2.5}}\right)$, while in \cite{tripuraneni2018stochastic}, the algorithm uses Hessian-vector products instead of Hessian estimates and the total number of Hessian-vector products required in \cite{tripuraneni2018stochastic} is $O\left(\frac{1}{\eps^{3.5}}\right)$.
	\item The total sample complexity for the \emph{zeroth-order} oracle in both the first-order and second-order SARC algorithms is $O\left(\frac{1}{\eps^{4.5}}\right)$.
	In contrast, the algorithm in \cite{tripuraneni2018stochastic} does not utilize function value estimates. While their algorithm benefits from not requiring function value estimates, the SARC algorithms offer the advantages of being adaptive and operate under weaker assumptions regarding the first- and second-order oracles. Specifically, SARC algorithms allows the estimates to be biased and can be arbitrary corrupted in each iteration with a constant probability, while the gradient and Hessian estimates in \cite{tripuraneni2018stochastic} are assumed to have bounded variance and bounded error almost surely.
\end{itemize}
\end{remark}

	\bibliographystyle{alpha}
        \bibliography{ref-random_ls}

\end{document}